\documentclass[12pt]{amsart}
\usepackage{amssymb}

\newtheorem{theorem}{Theorem}
\newtheorem{lemma}{Lemma}[section]
\newtheorem{corollary}[lemma]{Corollary}

\newtheorem{proposition}[lemma]{Proposition}

\theoremstyle{definition}
\newtheorem{remark}[lemma]{Remark}

\numberwithin{equation}{section}

\newcommand{\cal}{\mathcal}

\newcommand{\eps}{{\varepsilon}}

\newcommand{\Q}{{\mathbf Q}}
\newcommand{\R}{{\mathbf R}}
\newcommand{\Z}{{\mathbf Z}}

\begin{document}
\title {
Topological aspects of the Dvoretzky Theorem}

\author{Dmitri Burago}
\address{Dmitri Burago: Pennsylvania State University,
Department of Mathematics, University Park, PA 16802, USA}
\email{burago@math.psu.edu}

\author{Sergei Ivanov}
\address{Sergei Ivanov: 
St.Petersburg Department of Steklov Mathematical Institute,
Fontanka 27, St.Petersburg 191023, Russia}
\email{svivanov@pdmi.ras.ru}

\author{Serge Tabachnikov}
\address{Serge Tabachnikov: Pennsylvania State University,
Department of Mathematics, University Park, PA 16802, USA}
\email{tabachni@math.psu.edu}

\thanks{The first author was partially supported
by NSF grants DMS-0604113 and DMS-0412166.
The second author was partially supported by the
Dynasty foundation and  RFBR grant 08-01-00079-a.
The third author was partially supported by an NSF grant DMS-0555803.
}


\begin{abstract}
We explore possibilities and limitations of a purely topological approach to
the Dvoretzky Theorem.
\end{abstract}

\maketitle

\section*{Introduction} \label{intro}

The celebrated Dvoretzky theorem \cite{Dv} states that, for every $n$, any centered convex 
body of sufficiently high dimension has an almost spherical $n$-dimensional central section.
The purpose of this paper is to discuss a topological approach to this theorem.
We do this by considering a conjectural {\it Non-Integrable Dvoretzky Theorem} (see below).
We are definitely not the first ones who thought about this. M. Gromov
thought of the topological approach already in the sixties, but got discouraged by Floyd's 
examples which suggested that the approach would not work in dimensions greater
than two (a private communication). Apparently inspired by Gromov's ideas, V. Milman
presented a purely topological proof of the Dvoretzky Theorem for two dimensional sections
in the eighties (and gave explicit estimates) in \cite{Mi}.
The question was later extensively studied by V. Makeev. We 
confirm that indeed the topological approach does not work in all odd dimensions and 
also in dimension 4. It is often the case that people do not publish negative results
and in particular counterexamples, hence we cannot be certain that all results in this
paper are original, however it seems that the counterexamples have not been known before.
In a sense, the main contents of this paper is that a certain very natural approach is
a dead end. On the positive side, we show that topological approach does give interesting
generalizations even in higher dimensions.

Let us proceed with a description of the Dvoretzky Theorem and its non-integrable variants.
One defines a metric on the space of convex bodies containing 
the origin in their interiors in $n$-dimensional space:
$$
d(B_1, B_2)=\inf \log(s/t) \ \ {\rm such\ that}\ \exists s,t>0 : sB_1 \subseteq B_2 \subseteq tB_1;   
$$
here $sB$ denotes the homothetic image of $B$ with coefficient $s$. In particular, this metric is
defined for symmetric convex bodies. We use this metric for all counterexamples constructed in
this paper. When proving positive results, and especially for non-symmetric bodies or bodies
which may not contain the origin, we use the (normalized) Hausdorff distance (see the beginning
of Section \ref{2D}). 

The Dvoretzky theorem reads as follows:

\begin{theorem} \label{Dvor}
For any $\eps >0$ and any positive integer $n$,
there exists a number $N$ such that for any convex body $X\subset \R^N$
containing the origin in its interior, there exists an $n$-dimensional section of $X$
through the origin whose distance  from an $n$-dimensional Euclidean ball is less than $\eps$.
\end{theorem}

One wonders to what extend this is a topological theorem. As a matter of motivation,
consider the following toy example.

The following theorem holds: {\it each ellipsoid  $E\subset \R^3$ has a round central section, that is, a section which is a circle}.  Here is a simple topological proof. A section of an ellipsoid is an ellipse, and if a section of $E$ by (oriented) plane $\pi$ is not a circle, assign to $\pi$ the direction of the major axis of this elliptical section. The set of oriented planes through the origin is the sphere, and one obtains a field of directions on $S^2$. For topological reasons, such a field must have a singular point, and a singularity corresponds to a round section of the ellipsoid $E$.

This argument applies, without a change, to a more general situation. Assume that an ellipse is continuously chosen in every oriented 2-dimensional subspace of $\R^3$. Then at least one of these ellipses is a circle. This is a genuine generalization: not every field of ellipses in 2-dimensional subspaces of $\R^3$ comes from the sections of a 3-dimensional ellipsoid.

The general set-up is as follows. Let $G (n,N)$ be the Grassmanian of $n$-dimensional subspaces 
in $\R^N$ and  $G_+ (n,N)$ the Grassmanian of $n$-dimensional oriented subspaces in $\R^N$. Denote by $E(n,N) \to G (n,N)$ and $E_+(n,N) \to G_+(n,N)$ the tautological $n$-dimensional vector bundles. 

A natural question arises (see \cite{Ma1,Ma2}): {\it can the Dvoretzky Theorem be extended to an arbitrary continuous family of convex bodies in $E(n,N)$?} Such a conjectural extension will be referred to as the {\it Non-Integrable Dvoretzky Theorem}. One can also relax the symmetry requirements and ask whether every continuous family of convex bodies in $E(n,N)$ contains a body that is nearly symmetric with respect to some subgroup $G \subset SO(n)$. For a cyclic group $G=\Z_p$, this problem was extensively studied by V. Makeev, see his cited papers and references therein.

The Non-Integrable Dvoretzky Theorem holds for $n=2$, see \cite{Mi,Ma1,Ma2} and a proof in Section \ref{2D}. The main goal of this note is to construct counter-examples for greater values of $n$; namely, in Sections \ref{3D} and \ref{4D} we show that the Non-Integrable Dvoretzky Theorem does not hold for all odd $n$ and also for $n=4$. More formally:

\begin{theorem} \label{3Dex}
Let $n\ge 3$ be an odd number. Then there exists $\delta >0$ such that for every
$N\ge n$ there exists a continuous choice of symmetric convex bodies in $n$-dimensional
subspaces of $\R^N$, whose distances from $n$-dimensional Euclidean balls are greater than $\delta$.
\end{theorem}

The proof is essentially the same for all odd values of $n$. This is different for even-dimensional
sections, for our arguments are based on a concrete description of subgroups of low dimensional
Lie groups, and the following result is specifically 4-dimensional (even though we suspect that
it holds in all even dimensions):

\begin{theorem} \label{4Dex}
There exists $\delta >0$ such that for every $N\ge 4$ there exists a continuous choice
of symmetric convex bodies in oriented 4-dimensional subspaces of $\R^N$,
whose distances  from $4$-dimensional Euclidean balls are greater than $\delta$.
\end{theorem}

On the positive side, in Section \ref{cent} we show that if a polynomial of a fixed odd degree is assigned to every $n$-dimensional subspace of $\R^N$ then, for $N$ large enough, there exists a subspace to which there corresponds the zero  polynomial. This is a non-integrable version of a known result when the family of polynomials is obtained by restricting a single polynomial of $N$ variables to $n$-dimensional subspaces, see \cite{Bi}. We also show in Section \ref{cent} that every continuous family of convex bodies in $E(n,N)$, for $N$ large enough, contains a body that is nearly centrally symmetric. This result is not new: it is proved differently  in \cite{Ma1}. 

Finally, in Section \ref{toricsec}, we show that every continuous family of convex bodies in the fibers of $E_+(n,N)$, for $N$ large enough, contains a body that is nearly symmetric with respect to the maximal torus of the group $SO(n)$. This result is probably new.

Even though the counter-examples presented in this papers show that a purely topological
approach to the Dvoretzky Theorem has very serious limitation, one could ``through in"
some assumption on the modulus of continuity of the family of convex bodies in question.
Since the constructions are based on lifting maps by skeletons in cell complexes, it is
possible that even a mild assumption on the modulus of continuity could save topological
methods; however, the authors did not explore this direction at all.

{\bf Acknowledgments}.  We are grateful to J.C. Alvarez, M. Gromov, 
V. Makeev, V. Milman, and Yu. Zarhin for helpful discussions and historical and 
literature references.

\section{Preparations} \label{prelim}

In this section, we prepare tools needed for constructions of counter-examples in the
following two sections. 

Rather than dealing with general convex bodies, we deal only with convex bodies whose
supporting functions are given by spherical polynomials. It is not difficult to see that this
is not restrictive: indeed, any continuous function can be approximated by spherical
polynomials; on the other hand, given a spherical polynomial $p$, one can see that
the spherical polynomial $1+\eps p$ is  the support function (or radial function) of a 
convex body for all 
sufficiently small $\eps$. Let us proceed with details. 

Let $P^d_n$ denote the space of spherical polynomials of degree at most $d$
on the sphere $S^{n-1}\subset\R^n$. That is, an element of $P^d_n$ is a function
on $S^{n-1}$ obtained as a restriction of a polynomial of degree at most $d$
from $\R^n$ to $S^{n-1}$. Note that the restrictions of different polynomials
in $\R^n$ can give the same spherical polynomial. We regard $P^d_n$
as a subspace of $L^2(S^{n-1})$ and equip it with the induced Euclidean structure.

Let $F^d_n$ denote the subspace of $P^d_n$ consisting of even
(that is, symmetric) functions with zero average on $S^{n-1}$
(in other words, orthogonal to the one-dimensional subspace of constants).
Let $S(F^d_n)$ be the unit sphere in $F^d_n$.
The group $O(n)$ naturally acts on these spaces by isometries.

The meaning of the next lemma is that, for every $k$, there is an 
 canonical (that is $O(n)$-equivariant) map which produces a non-constant spherical
polynomial from a $k$-tuple of non-constant spherical polynomials.

 Let us recall that $A*B$ denote the join of topological spaces $A$ and $B$, 
 that is, the quotient space of $A\times B \times [0,1]$ by the equivalence 
 relation $(a,b_1,0)\sim (a,b_2,0)$ and $(a_1,b,1)\sim (a_2,b,1)$.  If $A$ 
 and $B$ are $G$-spaces, so is $A*B$. Denote $A^{(k)}=A*\dots*A$ ($k$ times).

\begin{lemma}\label{2poly}
For every even $d$ and every positive integer $k$,
there exists an $O(n)$-equivariant map
from $S(F^d_n)^{(k)}$ to $S(F^{kd}_n)$.
\end{lemma}

\begin{proof}
An element of the joint $S(F^d_n)^{(k)}=S(F^d_n)*\dots*S(F^d_n)$
is a tuple $(t_1f_1,t_2f_2,\dots,t_kf_k)$ where $f_i\in S(F^d_n)$,
$t_i\in[0,1]$ and $\sum t_i=1$.
We associate to such a tuple a spherical polynomial
$$
 g = (1+t_1f_1)(1+t_2f_2)\dots(1+t_kf_k)
$$
of degree at most $kd$.
We claim that $g$ is non-constant. If so, we can project it to $F^{kd}_n$ and then
scale it so that the resulting spherical polynomial lies in $S(F^{kd}_n)$. This provides us
with a desired equivariant map.

The above claim follows from the next statement: {\it let $g_1,\dots,g_k$ be non-zero even 
spherical polynomials on $S^{n-1}$, at least one of which is non-constant;
then $g=g_1 \dots g_k$ is non-constant as well}.
Indeed, extend $g_i$'s to $\R^n$ as homogeneous polynomials.
Such a homogeneous extension can be obtained
from an arbitrary extension by multiplying its monomials by appropriate powers
of ${x_1^2+\dots +x_n^2}$ where $x_i$ are the coordinate variables.
Reasoning by contradiction, assume that $g$ equals a constant $C$ on the sphere,
then $g_1\dots g_k=C\cdot(x_1^2+\dots +x_n^2)^q$ for some nonnegative integer~$q$.
Since the ring $\R[x_1,\dots,x_n]$ is a unique factorization domain
and the sum of squares is irreducible,
one concludes that each $g_i$ has a form $C_i\cdot(x_1^2+\dots +x_n^2)^{q_i}$
where $C_i$ is a constant and $q_i$ is a nonnegative integer.
Hence each $g_i$  is constant on the sphere, a contradiction.
\end{proof}

\begin{lemma} \label{mapping}
Let $M$ be a smooth closed manifold smoothly acted upon by a compact group
 $G\subset O(n)$, such that the stabilizer of no point is transitive on  $S^{n-1}$
 (with respect to the natural action of $G\subset O(n$) on the unit sphere).
 Then, for some even $d$, there exists a $G$-equivariant mapping $M \to S(F^d_n)$.
\end{lemma}
 
\begin{proof}
For every  $G$-orbit $P$ in $M$, consider a sufficiently small open
$G$-invariant tubular neighborhood $U$
(see \cite{Br} for the existence of such tubes).
If $U$ is small enough, there exists a $G$-equivariant retraction
$\pi_P:U\to P$ (for instance, let $\pi_P(x)$ be the nearest to $x$ point of $P$
with respect to a $G$-invariant Riemannian metric on $M$).

We obtain an open covering of $M$ by tubes $U$;
choose a finite sub-covering by neighborhoods $U_i,\ i=1,\dots,k$, of orbits $P_i$.
Let $\phi_i,\ i=1,\dots,k$ be a $G$-invariant partition of unity associated
with this covering $\{U_i\}$.
Let $G_i\subset G$ be the stabilizer of a point for $P_i$.
Choose a point $x_i \in P_i$ and an even non-constant spherical
polynomial $f_i$ on $S^{n-1}$, invariant under $G_i$;
this is possible since $G_i$ is not transitive on $S^{n-1}$.
Let $g_i\in S(F^d_n)$ be the normalized projection of $f_i$
to $F^d_n$. Here $d$ is the maximum degree of polynomials $f_i$.
Acting by $G$ on $x_i$ and on $g_i$, we  extend the correspondence
$x_i \mapsto g_i$ to an $G$-equivariant map $h_i:P_i\to S(F^d_n)$.
Now define a $G$-equivariant map $h:M\to S(F^d_n)^{(k)}$ by
$$
 h(x) = (\phi_1(x)h_1(\pi_{P_1}(x)), \dots, \phi_k(x)h_k(\pi_{P_k}(x)) ) .
$$
Applying an equivariant map from Lemma \ref{2poly} completes the proof.
\end{proof}

Let $G$ be a compact Lie group and $EG \to BG$ the universal principal $G$-bundle. 
For a $G$-space $F$, let $EF \to BG$ be  the associated bundle with fiber $F$.

\begin{lemma} \label{genus}
Suppose that there exists a $G$-equivariant and homotopically trivial map $f: F \to F$.
Then the bundle $EF^{(2)} \to BG$ admits a continuous section.
\end{lemma}

\begin{proof}
 Consider the mapping telescope (infinite iterated mapping cylinder) $X$ of the map  $f$, that is, the quotient space of 
$\coprod_i  F_i \times [0,1],\ i\in \Z$,
by the equivalence relation $((x)_i,1) \sim ((f(x))_{i+1},0)$, $x\in F$;
here $F_i$ denotes the $i$th copy of the space $F$
and $(x)_i$ the copy of a point $x\in F$ in $F_i$.
Since $f$ is homotopically trivial,
$X$ is homologically trivial. $X$ is naturally acted upon by the group $G$.

Define a $G$-equivariant map $g: X\to F*F$:
for $x\in F$ and $i\in\Z$ set
$$
g((x)_i,t)=
\begin{cases}
 (x,f(x),t) , &\text{$i$ is even}, \\
 (f(x),x,1-t) , &\text{$i$ is odd}.
\end{cases}
$$
Since $(x,f(x),1)=(f(f(x)),f(x),1)$ and $(f(x),x,0)=(f(x),f(f(x)),1)$ in $F*F$,
this map is well-defined and continuous.
Since $X$ is  homologically trivial, the bundle $EX \to BG$ has a section. The composition of this section  with the mapping $g$ yields a section of the bundle $EF^{(2)} \to BG$.
\end{proof}

\begin{remark}
Recall the notion of the Svarc genus $g(E)$ (or the sectional category) of a fiber bundle $E \to B$
with a fiber $F$:  the Svarc genus is the minimal number of open sets that cover the base $B$,
such that there exists a section of the bundle over each of these sets (see \cite{Sv,Ja}).
By one of the basic theorems in \cite{Sv}, $g(E)\leq k$ if and only if the
associated bundle with fiber $F^{(k)}$ has a section. It follows that,
under the assumptions of Lemma \ref{genus}, one has $g(EF)\leq 2$.
\end{remark}

\section{Odd-dimensional counterexample} \label{3D}

The goal of this section is to prove Theorem \ref{3Dex},
the negation of the conjectural Non-Integrable Dvoretzky Theorem in odd dimensions. 

Recall a remarkable construction, due to E. Floyd and described in \cite{C-M}. Consider the irreducible linear action of $SO(3)$ in $\R^5$; this action leaves invariant the unit sphere $S^4$. Floyd constructed an $SO(3)$-equivariant map $f:S^4 \to S^4$ whose degree is zero.

In our situation, we take the space of traceless quadratic forms in $\R^3$ as the 5-dimensional space;
such a form is uniquely determined by its restriction to $S^2\subset\R^3$, hence this space
is $F_3^2$. Then the 4-dimensional sphere from Floyd's construction is $S(F_3^2)\simeq S^4$.

Floyd's construction was generalized to $SO(n)$ for all odd $n$ by
W.-C.\ Hsiang and W.-Y.\ Hsiang \cite{HH}.
Namely there exists a contractible $SO(n)$-equivariant map $f:S(F_n^2)\to S(F_n^2)$,
cf.\ \cite[pp.~716--717]{HH}.

Applying Lemma \ref{genus} with $G=SO(n)$, $BG=G_+(n,\infty)$, $F=S(F_n^2)$
yields a section of the bundle $ES(F_n^2)^{(2)}\to BG=G_+(n,\infty)$.
By Lemma \ref{2poly} there is an $SO(n)$-equivariant map from $S(F_n^2)^{(2)}$
to $S(F_n^4)$, hence the bundle $ES(F_n^4)\to BG=G_+(n,\infty)$ also admits a section.
This means that on every oriented $n$-dimensional subspace $V\subset\R^N$ (with  $N$  arbitrarily large)
one can choose (in a continuous fashion) a non-constant even spherical polynomial $\Phi_V$
of degree at most~4, defined on the $(n-1)$-sphere $S(V)\subset V$.

For a non-oriented $n$-dimensional subspace $W\subset\R^N$, this gives us
an unordered pair $\{\Phi_{W^+},\Phi_{W^-}\}$ of non-constant spherical
polynomials on $S(W)$ where $W^+$ and $W^-$ are the two orientations of $W$.
Arguing as in the proof of Lemma \ref{2poly}, let $\Psi_W$ be the normalized
projection of the product $\Phi_{W^+}\Phi_{W^-}$ to the orthogonal complement
of the subspace of constants. This yields a continuous family $\{\Psi_W\}_{W\in G(n,N)}$
of spherical polynomials (defined on the corresponding $(n-1)$-spheres $S(W)$);
each of them has degree at most 8, zero average and unit $L^2$-norm.

Since $S(F_n^8)$ is compact, the values and first and second derivatives
of the functions $\Psi_W$ are uniformly bounded
(by a constant depending only on $n$). Hence for some $\eps=\eps(n)$
the bodies in subspaces $W$ given by the radial functions $1+\eps\Phi_W$ are convex.
Clearly, these bodies are uniformly separated from round balls
(with respect to the metric $d$).
This finishes the proof of Theorem \ref{3Dex}.

\begin{remark}
In the case $n=3$, a version of this construction yields a family of  convex
sets in 3-dimensional subspaces of $\R^N$ which are centrally symmetric
polyhedra with at most 24 vertices. Here is a description.

Let $Q$ be a traceless quadratic form in a 3-dimensional space $V$,
and let $\lambda, \mu\in\R$ be its eigenvalues having the same sign;
the third eigenvalue is $\nu=-\lambda-\mu$.
Assume first that $Q$ has a simple spectrum, and let $e_1, e_2, e_3$
be the respective unit eigenvectors.
Denote by $C(Q)$ the octahedron which is the convex hull
of the three pairwise orthogonal segments centered at the origin,
having the directions $e_1, e_2, e_3$ and the lengths
$(\lambda-\mu)^2, (\lambda-\mu)^2$ and $\nu^2$, respectively.
The correspondence $Q \mapsto C(Q)$ extends to quadratic forms with multiple spectrum,
that is, to the case when $\mu=\lambda$:
the corresponding octahedron $C(Q)$ then degenerates to a segment.
If $Q=0$ then this segment shrinks to the origin.

The first step of the above proof of Theorem \ref{3Dex} constructs
a section of the bundle $ES(F_3^2)^{(2)}\to BSO(3)=G_+(3,\infty)$.
This means that on every oriented $3$-dimensional subspace
$V\subset\R^N$ (with $N$ arbitrarily large) one has
a pair $(tQ_1,(1-t)Q_2)$ where $t\in[0,1]$ and
$Q_1,Q_2$ are unit-norm traceless quadratic forms on~$V$;
this pair depends continuously on~$V$.
Now assign to every such pair the convex hull the octahedra $C(tQ_1)$ and $C((1-t)Q_2)$.
This provides a continuous field of convex polyhedra with at most 12 vertices
in oriented 3-dimensional subspaces of $\R^N$. In a non-oriented subspace,
take the convex hull of the two polyhedra associated to the corresponding
oriented subspaces.

These polyhedra may degenerate (that is, some of them are planar polygons or line segments)
but they never degenerate to a point. A family of convex bodies (with nonempty ineriors)
can be obtained by taking a neighbourhood of radius~1 of each polyhedron.
\end{remark}

\section{4-dimensional counterexample} \label{4D}

In this section we construct a somewhat more sophisticated example for $n=4$.

Consider the standard linear representation of $SO(4)$ by orthogonal transformations of $\R^4$. 
Consider the 6-dimensional space of bivectors $\Lambda^2 \R^4$. One has the operation $*$ 
on bivectors defined by the formula: $a \wedge b = (*a,b)\  \Omega$ where $\Omega$ is the 
volume form in $\R^4$ and $(\ ,\, )$ is the scalar product. Since $*$ is an involution, 
$\Lambda^2 \R^4$ decomposes as $E_+ + E_-$, where $E_{\pm}$ are the 3-dimensional 
eigenspaces of $*$ with eigenvalues $\pm 1$.  Thus one has two homomorphisms, 
$\rho_{\pm}: SO(4) \to SO(3)$. 

In each space  $E_{\pm}$ consider the 5-dimensional space of traceless quadratic forms, as 
in Section \ref{3D}, and let $S_{\pm}$ be the respective unit 4-dimensional spheres acted 
upon by the group $SO(3)$. Let $M=S_+ \times S_-$. This is an $SO(4)$ space via the 
homomorphism $\rho_+ \times \rho_-: SO(4)\to SO(3)\times SO(3)$, and one has a homotopically 
trivial $SO(4)$-equivariant map $f \times f: M \to M$, the product of the Floyd maps described in 
Section \ref{3D}. 

As before, we use Lemma \ref{genus} to obtain a section of the bundle $EM^{(2)} \to BSO(4)$. We wish to apply Lemma \ref{mapping}, and we need the next result.

\begin{lemma} \label{isotr}
The isotropy subgroups of points under the $SO(4)$-action on $M$ are not transitive on $S^3$. 
\end{lemma}

\begin{proof}
Assume that  a subgroup  $G\subset SO(3)$ preserves a unit norm traceless quadratic form in $\R^3$. Two of the eigenvalues of this form $\lambda, \mu$ have the same sign, and the third eigenvalue is $\nu=-\lambda-\mu$ and has the opposite sign. The axis of the respective quadratic form, corresponding to the eigenvalue $\nu$, is invariant under $G$. Thus $G$ has an invariant vector.
 
Assume now that  a subgroup $H \subset SO(4)$ preserves a pair of unit norm traceless quadratic forms in $E_+$ and $E_-$.  Then $H$ preserves a pair of bivectors, say, $\omega_+ \in E_+$ and $\omega_- \in E_-$. We claim that $H$ has an invariant 2-plane in $\R^4$, and hence is not transitive on $S^3$.

One has a correspondence between 2-planes in $\R^4$ and decomposable bivectors, considered up to scalar factor, that is, the bivectors $\sigma$ such that $\sigma \wedge \sigma = 0$. The correspondence is as follows: given a 2-plane, choose a basis $(e,f)$ in it, and let $\sigma = e \wedge f$ (cf. \cite {Gl}).

We claim that the bivector $\sigma = \omega_+ + \omega_-$ is decomposable.  Indeed,
$$
\omega_+ \wedge \omega_+ = (\omega_+,\omega_+) \Omega=  \Omega,\quad  \omega_- \wedge \omega_- = (-\omega_-,\omega_-)  \Omega = - \Omega,$$
 and  $\omega_+ \wedge \omega_- = 0.$ Hence $\sigma \wedge \sigma =0$.
It follows that the 2-plane, corresponding to the bivector $\sigma$, is $H$-invariant, as claimed.
\end{proof}

To complete the construction, we now apply Lemma \ref{mapping} and obtain an
$SO(4)$-equivariant map $g:M\to S(F_4^d)$ for some even $d$.
Then there is an $SO(4)$-equivariant map $g*g:M^{(2)}\to S(F_4^d)^{(2)}$
and hence (by Lemma \ref{2poly}) an $SO(4)$-equivariant map $h:M^{(2)}\to S(F_4^{2d})$.
Since the bundle $EM^{(2)} \to BSO(4)$ admits a section, $h$ yields a section
of the bundle $ES(F_4^{2d})\to BSO(4)$. This means that we have associated
to each fiber $F$ of the tautological bundle $E_+(4,N)\to G_+(4,N)$ a non-constant
even spherical polynomial of degree $2d$ on the 3-sphere $S(F)$.
Proceeding as in Section \ref{3D} yields a family of symmetric convex bodies
uniformly separated away from round balls.
This completes the proof of Theorem \ref{4Dex}.

\section{Polynomial approximations and 2-dimensional case} \label{2D}

In this section, we develop certain tools needed in the sequel. We also apply these tools here
to present here yet another proof of the Non-Integrable Dvoretzky Theorem for two-dimensional
sections and derive some corollaries with effective estimates; our argument consider a slightly 
more general case, namely the bodies maybe degenerate and are not required to contain 
the origin. 

Since we consider families of convex bodies that are not required to contain the origin, in the sequel
we use the Hausdorff distance in place of $d$ described in the introduction. To normalize it, we
assume that all bodies in the families are contained in the unit ball centered at the origin. Recall
that the Hausdorff distance $d_h$ between compact sets $B_1$ and $B_2$ in $R^n$ is given by the
formula:
$$
d_h(B_1, B_2)=\inf \{r>0 \  {\rm such\ that} \  B_1 \subseteq U_r(B_2)  \ {\rm and}\  
B_2  \subseteq U_r(B_1) \},   
$$
where $U_r(A)$ denotes the $r$-neighborhood of a set $A$.
The domain of $d$ (the class of convex bodies containing the origin in their interior) 
is open with respect to $d_h$.
One easily sees that $d$ and $d_h$ define the same topology
on this domain.
Moreover, if $B$ is the unit ball centered at the origin,
then $d(A,B)/d_h(A,B)\to 1$ as $A\to B$ (w.r.t.\ $d$ or $d_h$).

One can characterize a compact convex set $A\subset\R^n$ by its support function
$h_A:S^{n-1} \to \R$,
the signed distance from the origin to the support hyperplane orthogonal to a given direction.
More precisely,
$$
 h_A(v) = \sup\{ \langle v,x\rangle : x\in A \}, \qquad v\in S^{n-1}
$$
where $\langle,\rangle$ is the scalar product in $\R^n$.

Recall that $P_n^d$ denotes the space of spherical polynomials
of degree at most $d$, regarded as a subspace of $L^2(S^{n-1})$.
Denote by $\pi_n^d:L^2(S^{n-1})\to P_n^d$ the orthogonal
projection to this subspace.
The following lemma allows us to study symmetries of spherical polynomials
instead of those of convex sets.

\begin{proposition}\label{harmonics}
Let $G \subset O(n)$ be a compact subgroup of the orthogonal group.
For every $n\ge 2$ and every $\eps>0$ there exists a positive integer $d$
such that the following holds.
If a compact convex set $A\subset\R^n$ is contained in the unit ball
centered at the origin and the function $\pi_n^d(h_A)$ is $G$-invariant,
then $A$ lies within $d_h$-distance $\eps$ from a $G$-invariant convex set. 
\end{proposition}

\begin{proof}
We need the following well-known properties of support functions, cf. e.g. \cite{Sc}:

(i) The Hausdorff distance between convex sets is equal to the $C^0$ distance 
between support functions, namely:
$$
 d_h(A,B) = \|h_A-h_B\|_{C^0} = \sup_{v\in S^{n-1}} |h_A(v)-h_B(v)|
$$
for every pair of compact convex sets $A$ and $B$.

(ii) A function on the sphere $S^{n-1} \subset \R^n$ is the support function of a
convex set if and only if its positive homogeneous extension to $\R^n$ is a
convex function. As a corollary, a function which is a limit of a sequence of 
support functions is a support function as well. 

Let $\cal H$ denote the set of support functions of all compact convex subsets
of the unit ball in $\R^n$. One easily sees that every function from $\cal H$
is 1-Lipschitz. Therefore $\cal H$ is compact in $C^0$ and hence in $L^2$.

Since spherical polynomials are dense in $L^2(S^{n-1})$,
for every $f\in \cal H$
we have $\pi_n^d(f)\to f$ in $L^2(S^{n-1})$ as $d\to\infty$.
The following lemma asserts that this convergence is uniform on $\cal H$.

We need the following two standard technical lemmas:

\begin{lemma}\label{L2-uniform}
For every $\eps>0$ there is a positive integer $d=d(n,\eps)$
such that ${\|f-\pi_n^d(f)\|_{L^2}<\eps}$ for all $f\in\cal H$.
\end{lemma}

\begin{proof}
Suppose the contrary.
Then there is an $\eps>0$ such that for every integer $d>0$
there is a function $f_d\in\cal H$ such that
$$
 \|f_d-\pi_n^d(f_d)\|_{L^2}\ge \eps.
$$
Note that the left-hand side of this inequality
is the $L^2$-distance from $f_d$ to the subspace $P_n^d$.
Thus $f_d$ is $\eps$-separated from $P_n^d$ in $L^2$.
By compactness, a subsequence $\{f_{d_i}\}$ of $\{f_d\}$
converges to a function $f\in\cal H$ in $L^2$.
Since the subspaces $P_n^d$ are nested and
$f_{d_i}$ is $\eps$-separated from $P_n^{d_i}$,
one concludes that $f$ is $\eps$-separated from $P_n^d$ for all~$d$.
This contradicts  to the density of spherical polynomials in $L^2$.
\end{proof}

\begin{lemma}\label{C0viaL2}
There is a constant $C=C(n)$ such that
$$
 \|f-g\|_{C^0}\le C \|f-g\|_{L^2}^{2/(n+1)}
$$
for all $f,g\in\cal H$.
\end{lemma}

\begin{proof}
Recall that $f$ and $g$ are 1-Lipshitz and 
and their absolute values are bounded above by~1.
Denote $h=f-g$ and $a=\|h\|_{C^0}$.
Then $h$ is 2-Lipschitz and $a\le 2$.
Let $x_0\in S^{n-1}$ be a point such that $|h(x_0)|=a$.
Then $|h(x)|\ge a/2$ for all $x\in S^{n-1}$
lying within the ball of radius $a/4$ centered at $x_0$.
The $(n-1)$-volume of this ball in $S^{n-1}$ is bounded from
below by $C_1a^{n-1}$ for a suitable constant $C_1=C_1(n)>0$.
Therefore
$$
 \|h\|^2 = \int_{S^{n-1}} h^2 \ge C_1a^{n-1} (a/2)^2 = C_1 a^{n+1}/4 .
$$
Hence
$$
 \|h\|_{C^0} = a \le C\|h\|^{2/(n+1)}
$$
for $C=(4/C_1)^{1/(n+1)}$.
\end{proof}

Now return to the proof of Proposition \ref{harmonics}.
For an $\eps>0$, choose $d=d(n,\eps_1)$ from Lemma \ref{L2-uniform}
for $\eps_1=\frac12(\eps/C)^{(n+1)/2}$ where
$C$ is the constant from Lemma \ref{C0viaL2}.
Now let $A$ be a convex set in the unit ball in $\R^n$ and
denote  $f=h_A\in\cal H$.
Assume that $g=\pi_n^d(f)$ is $G$-invariant,
then for every $\gamma\in G$ we have
$\pi_n^d(f\circ\gamma)=g\circ\gamma=g$,
hence
$$
 \|f-f\circ\gamma\|_{L^2} \le \|f-g\|_{L^2}+\|f\circ\gamma-g\|_{L^2} \le 2\eps_1
$$
by the choice of~$d$. Then by Lemma \ref{C0viaL2},
\begin{equation}
\label{C0}
 \|f-f\circ\gamma\|_{C^0} \le (2\eps_1)^{2/(n+1)} =\eps .
\end{equation}
Now consider a function $F$ on $S^{n-1}$ defined by
$$
 F(x) = \int_G f\circ\gamma(x) \, d\mu(\gamma) 
$$
where $\mu$ is the probability Haar measure on $G$.
Since all functions $f\circ\gamma$ have homogeneous convex
extensions, so does $F$. Therefore $F$ is the support function
of a convex set $B$.
The definition of $F$ and \eqref{C0} imply that
$\|f-F\|_{C^0} \le \eps$ and hence $d_h(A,B)\le\eps$.

This completes the proof of Proposition \ref{harmonics}.
\end{proof}

From now on we study two-dimenional sections.
Let ${\cal F}$ be a continuous family of (non-empty) compact convex sets
in the fibers of the bundle $E_+(2,N) \to G_+(2,N)$.

The following proposition implies that if $N$ is large enough then every family ${\cal F}$
contains a nearly round disk centered at $O$.

\begin{proposition} \label{Fourier}
For each $d$ and $N \geq d$, every  continuous family ${\cal F}$ of compact convex sets
in the fibers of the bundle $E_+(2,N) \to G_+(2,N)$ contains a set whose support function 
has a constant  Fourier polynomial of degree $d$, i.e., is free of the first $d$ harmonics.
\end{proposition}

\begin{proof}
For a positive integer $q$,  
denote by $H^q_2$ the subspace of $P^q_2$ spanned by the functions
$\cos (q\alpha)$ and $\sin (q\alpha)$ where $\alpha$ is the angular coordinate in $S^1$.
Then the subspace $\overline P^d_2:=\bigoplus_{q=1}^d H^q_2$ of $P^d_2$ is the orthogonal
complement to the subspace of constants.
Since $SO(2)=U(1)$, we may consider  $E_+(2,N) \to G_+(2,N)$ as a 1-dimensional 
complex bundle, $\xi$. The first Chern class $c_1(\xi)$ is the Euler class $e$ of the bundle 
$E_+(2,N) \to G_+(2,N)$.

The bundle $\xi^q$ has the full Chern class  $c(\xi^q)=1+qe$; its sections are 
homogeneous complex polynomials of degree $q$. Considered as a real bundle, 
this is the bundle of $q$th harmonics  with fiber $H^q_2$. 

We claim that the bundle $\eta$ with fiber
 $\overline P^d_2$, associated with the bundle $E_+(2,N) \to G_+(2,N)$, has 
 no non-vanishing sections. Indeed,  $\eta=\sum_{q=1}^d \xi^q$, and its highest 
 characteristic class  is $d! e^d \in H^{2d}(G_+(2,N),\Z)$. If $N\geq d$, this  
 is a non-zero class (see, e.g, \cite{M-S}).

Finally, given a family ${\cal F}$, assign to each convex body the Fourier polynomial 
of degree
$d$ of its support function minus its free term. This yields a section of the bundle 
$\eta$ that must
have zeroes, and the result follows.
\end{proof}

\begin{corollary}
For each $d$ and $N \geq d$, every  continuous family ${\cal F}$ of non-degenerate
compact convex sets
in the fibers of the bundle $E_+(2,N) \to G_+(2,N)$ contains a body whose boundary has
at least $2d+2$ intersections (counted with multiplicities) with a circle centered at the origin.
\end{corollary}

\begin{proof} Consider the support function with a constant  Fourier polynomial of degree $d$
and subtract its constant term $c$.
The resulting function has no fewer   zeros than its first non-trivial harmonic, that is,
at least $2d+2$ zeros (see, e.g., \cite{O-T} on this Sturm-Hurwitz theorem).
Hence the boundary of the respective body intersects the circle of radius $c$ at least $2d+2$
times (the fact that $c>0$ follows from two facts: (i) if $O$ is inside the body then the support
function is positive, and therefore its average value, $c$, is positive as well;
(ii) parallel translating the origin results in adding a first harmonic to the support function,
which does not change the average value).  
\end{proof}

As another corollary, we have the following non-integrable 2-dimensional Dvoretzky theorem.

\begin{theorem}
For any $\eps >0$, there exists a number $N$ such that for any $n\geq N$ and any  
 continuous family ${\cal F}$ of compact convex sets
in the fibers of the bundle $E_+(2,N) \to G_+(2,N)$
there exists $F \in {\cal F}$ which is either a single point at the origin
or contains the origin in the interior and lies within $d$-distance $\eps$
from a  Euclidean disc centered at the origin.
\end{theorem}

\begin{proof}
We may assume that the sets are scaled so that the maxima of
their support functions are equal to 1. Then Propositions \ref{harmonics}
and \ref{Fourier} imply the statement with $d_h$ in place of~$d$.
Then the theorem follows from the fact that $d$ and $d_h$
define the same topology in a neighborhood of the unit disc.
\end{proof}

\section{Odd polynomials} \label{cent}

Let $d$ be an odd positive integer.
Consider the bundle $S^d(E(n,N))$ of homogeneous polynomials of degree $d$ 
and the bundle
$$
Q^d(E(n,N))=S^{1}(E(n,\infty)) \oplus S^{3}(E(n,\infty))\oplus\dots\oplus S^{d}(E(n,\infty))
$$
of all odd polynomials of degree at most~$d$
associated  with the 
tautological $n$-dimensional bundle $E(n,N) \to G (n,N)$.

\begin{proposition} \label{odd}
For every odd $d$ and every $n$ there exists $N$ such that for every $n\geq N$ the bundle
$Q^d(E(n,N))$ has no non-vanishung sections.
\end{proposition} 

\begin{proof}
Using the splitting principle, write $E(n,\infty)=\xi_1 \oplus \dots \oplus \xi_n$
where $\xi_i$ are linear bundles with the full Stiefel-Whitney classes $w(\xi_i)=1+x_i$.
Then  $w_j(E(n,\infty))=\sigma_j(x)$, the elementary symmetric function of $x_1,\dots,x_n$.
Therefore
$$
S^d(E(n,\infty))=\sum_{\sum_k j_k=d,\  j_k\geq 0} \xi_1^{j_1} \otimes\dots\otimes \xi_n^{j_n},
$$
hence
$$
w(S^d(E(n,\infty)))=\Pi_{\sum_k j_k=d} \left(1+\sum j_k x_k\right),
$$ 
and the top Stiefel-Whitney class is
$$
P_d:=\Pi_{\sum_k j_k=d} \left(\sum j_k x_k\right).
$$
$P$ is a symmetric polynomial $x_j$ with coefficients in $\Z_2$. We claim that $P\neq 0$.
Indeed, if each $x_j=1$ then $P=d^M$ where 
$$
M=\binom{d+n-1}{n-1},
$$
the number of solutions of the equation $\sum j_k=d$.  
Since $d$ is odd, $P_d=1$.

The ring $H^*(G (n,\infty),\Z_2)$ is the ring of symmetric polynomials in variables
$x_1,\dots,x_n$, therefore $P_d\neq 0$.
Hence the top Stiefel-Whitney class $P=P_1P_3\dots P_d$ of $Q^d(E(n,N))$ is nonzero.
It follows that $P\neq 0$ in $H^*(G (n,N),\Z_2)$ for sufficiently large $N$. 
Hence the bundle $Q^d(E(n,N))$ has no non-vanishing sections.
\end{proof}

Proposition \ref{odd} implies the following non-integrable theorem that is already known due to V. Makeev, see \cite[Theorem 6]{Ma1}.

\begin{theorem} \label{sym}
For any $\eps >0$ and any $n$ there exists a number $N$ such that for any $n\geq N$ and any  
continuous family $\cal F$ of compact convex sets contained in the unit balls centered at the origin in
the fibers of the bundle $E (n,N) \to G (k,N)$ there 
exists a set $F \in {\cal F}$ whose $d_h$-distance  from a centrally symmetric set with respect 
to the origin is less than $\eps$.
\end{theorem} 

\begin{proof}
The proof proceeds along the same lines as in Section \ref{2D}. One characterizes a convex body
by its support function $S^{n-1}\to \R$.
The support function is approximated by a spherical polynomial,
a polynomial is  decomposed into the even and odd parts,
and Proposition \ref{odd} implies that the odd part vanishes for some body $F\in {\cal F}$. 
Applying Proposition \ref{harmonics} with $G=\{Id,-Id\}$ completes the proof.
\end{proof}

Another immediate consequence of Proposition \ref{odd} is a theorem on zero linear subspaces of odd polynomials.

\begin{theorem} \label{flat}
For every odd $d$ and a positive integer $n$ there exists $N$ such that the following holds.
For every continuous family $\{F_V\}$ of odd polynomial functions $F_V:V\to\R$ of degree at most~$d$, where $V$ ranges over
all $n$-dimensional subspaces of $\R^N$, there exists a subspace $V\subset \R^N$ such that $F_V\equiv 0$.
\end{theorem} 

In particular, one may start with a polynomial $F$ of odd degree on $\R^N$,
and let $F_V$ be the restriction of $F$ to the subspace $V$.
Then the assertion of Theorem  \ref{flat} is  due to Birch, see \cite{Bi,Wo} and \cite{A-H},
and to Debarre and Manivel \cite{D-M}.
Thus Theorem \ref{flat} is a non-integrable version of the Birch theorem.

\section{Toric symmetry} \label{toricsec}

In this final section, we prove a non-integrable Dvoretzky-type theorem (Theorem \ref{toric})
for symmetries with respect to the maximal torus $T^{[n/2]}\subset SO(n)$
(and therefore for all its subgroups). 
This generalizes the results of Makeev for the group $\Z_p$ (which actually implies
the toric symmetry for 1-dimensional torus). Of course, the two-dimensional Dvoretzky
Theorem is a particular case of Theorem \ref{toric}.

We say that a body $B \subset \R^n$ has a $T^{[n/2]}$-symmetry if it is invariant under 
the action of a maximum torus of $SO(n)$ (acting linearly on $\R^k$ in the
standard way), that is a subgroup of  $SO(n)$ conjugate to 
a standard torus  $T^{[n/2]}  \subset SO(n)$.

\begin{theorem} \label{toric}
For any $\eps >0$ and any $n$, there exists a number $N$
such that for any  continuous family of compact convex sets ${\cal F}$
in unit balls centered at the origin in the fibers of the bundle $E_+(n,N) \to G_+(n,N)$
 there exists $F \in {\cal F}$ whose $d_h$-distance from a set 
having a $T^{[n/2]}$-symmetry is less than $\eps$.
\end{theorem}

\begin{proof} Recall some material on cohomological theory of topological transformation groups, see \cite{Hs}.
Let $G$ be a compact group and $X$ a topological $G$-space. Let $EG\to BG$ be the universal
principal $G$-bundle. If the action of $G$ on $X$ is not free, the quotient space $X/G$ may be
topologically unsatisfactory. The following  construction of homotopical quotient is due to A. Borel:
$X_G=(E\times X)/G$ where the action of $G$ on $E\times X$ is diagonal.
The space $X_G$ fibers over $BG$ with fiber $X$; this fiber bundle is associated with $EG\to BG$. 
 The equivariant cohomology $H_G^*(X)$ are defined as $H^*(X_G)$.
In particular, $H_G^*(pt)=H^*(BG)$, and one has a homomorphism $\pi^*: H_G^*(pt) \to H_G^*(X)$
induced by the projection $\pi: X_G \to BG$.

We need the following result from \cite{Hs}.
Let $G=T^r$, and  let $X$ be  a paracompact $G$-space  with finite cohomology dimension. 

\begin{proposition}[\cite{Hs}, Chapter 4.1, Corollary 1]
\label{fix}
The fixed point set $X^G$ of the group $G$ in $X$ is non-empty if and only 
if $\pi^*: H_G^*(pt; \Q) \to H_G^*(X;\Q)$ is a monomorphism.
\end{proposition}

With this preparation, we are ready to prove the theorem. 
Let $G= T^{[n/2]}$ and $X=P^d_n - (P^d_n)^G$,
the set of spherical polynomials of degree at most $d$ with the $G$-symmetric polynomials deleted.
We claim that the bundle $\pi: X_G \to BG$ has no sections. Indeed, the action of $G$ on $X$
has no fixed points, therefore, by Proposition \ref{fix}, the homomorphism
$\pi^*: H^*(BG) \to H^*(X_G)$ has a non-trivial kernel. Hence no section exists.

Now consider the bundle $E_+X\to G_+(n,\infty)=BSO(n)$ with fiber $X$,
associated with the universal principal $SO(n)$-bundle $E_+(n,\infty)\to G_+(n,\infty)$.
Since $G\subset SO(n)$, the former bundle does not have sections.
It follows that the restriction of this bundle to $G_+(n,N)$, with $N$ large enough,
does not have sections either.
In other words, if $N$ is large enough, every section of the bundle $P^d(E)\to G_+(n,N)$
of spherical polynomials of degree at most $d$ assumes a value which is
a spherical polynomial enjoying  a $T^{[n/2]}$-symmetry.
Now the theorem follows from Proposition \ref{harmonics}.
\end{proof}

\end{document}